\documentclass{amsart}

\usepackage{graphicx}
\usepackage{enumerate}
\usepackage{amssymb}

\newtheorem{theorem}{Theorem}[section]
\newtheorem{lemma}[theorem]{Lemma}
\newtheorem{proposition}[theorem]{Proposition}

\newtheorem{mtheorem}{Theorem}

\newtheorem{mcorollary}[mtheorem]{Corollary}

\theoremstyle{definition}
\newtheorem{definition}[theorem]{Definition}
\newtheorem{remark}[theorem]{Remark}
\newtheorem{convent}[theorem]{Convention}

\numberwithin{equation}{section}
\numberwithin{figure}{section}

\begin{document}

\title[Invariant sets with and without SRB measures]{Coexistence of invariant sets with and without SRB measures in H\'enon family}

\author{Shin Kiriki} 
\address{Department of Mathematics, Kyoto University of Education, 
Fukakusa-Fujinomori 1, Fushimi, Kyoto, 612-8522, JAPAN}
\email{skiriki@kyokyo-u.ac.jp}

\author{Ming-Chia Li}
\address{Department of Applied Mathematics, 
National Chiao Tung University, Hsinchu 300, TAIWAN}
\email{mcli@math.nctu.edu.tw}

\author{Teruhiko Soma}
\address{Department of Mathematics and Information Sciences,
Tokyo Metropolitan University,
Minami-Osawa 1-1, Hachioji, Tokyo 192-0397, JAPAN}
\email{tsoma@tmu.ac.jp}

\dedicatory{
Dedicated to the memory of Floris Takens (Nov.\ 12, 1940 -- Jun.\ 20, 2010).
}


\date{\today}

\begin{abstract}
Let $\{f_{a,b}\}$ be the (original) H\'enon family.
In this paper, we show that, for any $b$ near $0$, there exists a closed interval $J_b$ which contains 
a dense subset $J'$ such that, for any $a\in J'$, $f_{a,b}$ has a quadratic homoclinic tangency 
associated with a saddle fixed point of $f_{a,b}$ which unfolds generically with respect to the one-parameter family $\{f_{a,b}\}_{a\in J_b}$.
By applying this result, we prove that $J_b$ contains a residual subset $A_b^{(2)}$ 
such that, for any $a\in A_b^{(2)}$, $f_{a,b}$ admits the Newhouse phenomenon.
Moreover, the interval $J_b$ contains a dense subset $\tilde A_b$ such that, for any $a\in \tilde A_b$, $f_{a,b}$ 
has a large homoclinic set without SRB measure and a small strange attractor with SRB measure simultaneously.
\end{abstract}

\maketitle

\section{Introduction}\label{Introduction}

In \cite{He} H\'{e}non studied numerically the dynamics of the diffeomorphisms $f_{a,b}:\mathbb{R}^2\to 
\mathbb{R}^2$ defined  as 
\begin{equation}\label{original}
f_{a,b}(x,y)=(1-ax^2+y, bx)
\end{equation}
for $b\neq 0$ and presented a supporting evidence for the existence of a strange attractor of $f_{a,b}$ when $a=1.4$ and $b=1.3$.
We call these maps \textit{H\'enon maps} or more strictly \textit{original} H\'enon maps
consciously to distinguish them from H\'enon-like maps.
Subsequently, Forn{\ae}ss and Gavosto \cite{FG} gave a computer assisted proof that, for some  $a_0=1.3924198\cdots$ and 
any fixed $b_0$ near $0.3$, $f_{a_0,b_0}$ has a generic unfolding homoclinic tangency.
Furthermore, applying results of Mora and Viana \cite{MV} to the H\'{e}non maps, one has that 
$f_{a,b}$ exhibits a small strange attractor for parameters $(a,b)$ arbitrarily close to $(a_0,b_0)$.

Let $M$ be a surface and 
$\mathrm{Diff}^r(M)$ the space of $C^r$ diffeomorphisms with $C^r$ topology for an integer $r\geq 2$. 
We say that a connected subset $A$ of $\mathrm{Diff}^r(M)$ has \textit{persistent homoclinic tangencies} if a continuation of basic sets $\Lambda(f)$ of $f\in A$ is well defined and each $\Lambda(f)$ has a {\it homoclinic tangency}.
That is, there exist points $x, y\in \Lambda(f)$, possibly $x=y$, such that $W^u(x)$ and  $W^s(y)$ have a tangency.
Newhouse \cite{N0,N1} (see also  Palis-Takens \cite{PT}) showed that, 
if $f\in \mathrm{Diff}^r(M)$ has a homoclinic tangency associated with a dissipative saddle  point, 
then for any neighborhood $U(f)$ of $f$ in $\mathrm{Diff}^r(M)$, there is an element $g\in U(f)$ some connected neighborhood $N$ of which has persistent homoclinic tangencies.
Moreover, there exists a residual subset of $N$ each element of which is a diffeomorphism admitting infinitely  many sinks.
The condition with infinitely many sinks is called the \textit{Newhouse phenomenon}.
Robinson \cite{R} detected the phenomenon in the context of certain one-parameter families in $\mathrm{Diff}^r(M)$.
One of our aims in this paper is to show that  the original H\'enon family admits the Newhouse phenomenon.

Let $\Lambda_1$ and $\Lambda_2$ be basic sets  for $f\in \mathrm{Diff}^r(M)$.
Then, we say that $\Lambda_1$ is \textit{homoclinically related} to  $\Lambda_2$ if either  $\Lambda_1=\Lambda_2$
or there are points $x_1, y_1 \in \Lambda_1$ and $x_2, y_2 \in \Lambda_2$ such that $W^u(x_1) \setminus \Lambda_1$ has a non-empty
transverse intersection with $W^s (x_2)  \setminus \Lambda_2$ and $W^s (y_1)  \setminus \Lambda_1$ has a non-empty transverse
intersection with $W^u(y_2)   \setminus \Lambda_2$. 
The closure $\Lambda$ of the union of basic sets homoclinically related to each other is called a \textit{homoclinic set} if $\Lambda$ contains more than a single periodic orbit.
When one of these basic sets contains a periodic point $p$ of $f$, $\Lambda$ is called the homoclinic set of $p$.
We say that a homoclinic set $\Lambda$ has a \textit{homoclinic tangency} if $\Lambda$ contains a basic set with a homoclinic tangency.
Thirty years after Newhouse's original works, new results were presented by himself.
In fact, he showed in \cite[Theorem 1.4]{N2} that there is a residual subset $R$ of $\mathrm{Diff}^r(M)$  such that, if
$f\in R$ and  $\Lambda(f)$ is a homoclinic set for $f$ which contains a homoclinic tangency and has an associated dissipative saddle point, then $\Lambda(f)$  does not carry an SRB measure.
Here, \textit{SRB measure} means  an 
$f$-invariant Borel probability measure which is ergodic, has  a compact support, and 
has absolutely continuous conditional measures on unstable manifolds.
Moreover,
Newhouse gave  the conjecture in  \cite{N2}:
\begin{itemize}
\item for each
parameter $b$, there is a residual set of parameters $a$ such that $f_{a,b}$ has no SRB measure.
\end{itemize}
On the other hand,
Benedicks and Young \cite{BY} showed 
\begin{itemize}
\item 
 for almost every positive $b$ near $0$,
there is  a positive Lebesgue measure set $A_b$ of $a$-parameters such that $f_{a,b}$ has an SRB measure supported by the homoclinic set of a fixed point of $f_{a,b}$ if $a\in A_b$. 
\end{itemize}

Following the situation of admitting versus non-admitting of SRB 
measures, we will show that these two conditions coexist for 
many parameter values in the original H\'enon family.
For convenience in our arguments, we adopt the following topologically conjugated formula of the H\'enon map $f_{a,b}$: 
$$
\varphi_{a,b}(x,y)=(y, a-bx+y^2) 
$$
which is obtained from the classical formula (\ref{original}) by the 
 reparametrization $(a,b)\mapsto(-a,-b)$ and the coordinate change $(x,y)\mapsto (-ab^{-1}y, -ax)$.
Note that $p_{a,b}=(y_{a,b},y_{a,b})\in \mathbb{R}^2$ with $y_{a,b}=(1+b+\sqrt{(1+b)^2-4a})/2$ 
is a fixed point of $\varphi_{a,b}$.
\medskip

Now we state our main results.

\begin{mtheorem}\label{thm_A}
There exists an open interval $I$ containing $0$ such that, for any $b\in I\setminus \{0\}$, 
there is a positive integer $w$ and a closed interval $J_b$  in the $a$-parameter space satisfying 
the following {\rm (i)} and {\rm (ii)}.
\begin{enumerate}[\rm (i)]
\item
For any $a\in J_b$, $\varphi_a:=\varphi^w_{a,b}$ has continuations of two basic sets 
$\Lambda^{\mathrm{out}}_a$ and $\Lambda^{\mathrm{in}}_a$ with $p_{a,b}\in \Lambda^{\mathrm{out}}_a$ 
such that there exist persistent quadratic tangencies of $W^u(\Lambda^{\mathrm{out}}_a)$ and $W^s(\Lambda^{\mathrm{in}}_a)$ which unfold generically with respect to the one-parameter family $\{\varphi_a\}_{a\in J_b}$.
\item
There is a dense subset $J'$ of $J_b$ such that, for any $\hat a\in J'$, 
$W^u(p_{\hat a,b})$ and $W^s(p_{\hat a,b})$ have a quadratic tangency $q_{\hat a}$ which unfolds generically 
with respect to $\{\varphi_a\}_{a\in J_b}$.
\end{enumerate}
\end{mtheorem}

See Section \ref{s_persistent} for the definition of persistent quadratic
tangency unfolding generically. 
A more detailed version of Theorem \ref{thm_A} (i) and (ii)
is stated as Theorems \ref{persistent tangency} and \ref{l_0}, respectively.
Together with results of
\cite{MV,BY,PT,R,N2,WY}, 
the two theorems also imply the
following.


\begin{mtheorem}\label{thm_B}
For any $b\in I\setminus \{0\}$, the interval $J_b$ given in Theorem \ref{thm_A} contains subsets $A_b^{(1)}$, $A_b^{(2)}$, $A_b^{(3)}$ satisfying 
the following conditions.
\begin{enumerate}[\rm (i)]
\item
$A_b^{(1)}$ is open dense in $J_b$ and, for any $a\in A^{(1)}_b$, $\varphi_{a,b}$ does not have an SRB measure supported by the homoclinic set of $p_{a,b}$.
\item
$A_b^{(2)}$ is a residual subset of $J_b$ with $A_b^{(2)}\subset A_b^{(1)}$ and, for any $a\in A^{(2)}_b$, $\varphi_{a,b}$ has infinitely many sinks.
\item
$A_b^{(3)}$ has Lebesgue measure positive everywhere in $J_b$ and, for any $a\in A^{(3)}_b$, $\varphi_{a,b}$ has an SRB measure supported by an H\'{e}non-like strange attractor.
\end{enumerate}
\end{mtheorem}

Note that the H\'{e}non-like strange attractor given in Theorem \ref{thm_B}~(iii) is a small 
invariant set which  arises from renormalization near the tangency $q_{\hat a}$ of 
Theorem \ref{thm_A}, 
while the homoclinic set in Theorem \ref{thm_B} (i) is a large invariant set.

We say that a subset $A$ of an interval $J$ has \emph{Lebesgue measure positive everywhere} 
 if, for any non-empty open subset $U$ of $J$, 
$A\cap U$ has positive Lebesgue measure. From the definition, we know that such a set $A$ is dense in $J$.
Also, an invariant set $\Omega$ of $\varphi_{a,b}$ is called a \emph{strange attractor} if (a) there exists a saddle point $p\in \Omega$ such that the unstable manifold $W^u(p)$ has dimension $1$ and $\mathrm{Cl}(W^u(p))=\Omega$, (b) there exists an open neighborhood $U$ of $\Omega$ such that $\{f^n(U)\}_{n=1}^\infty$ is a decreasing sequence with $\Omega=\bigcap_{n=1}^\infty f^n(U)$, and (c) there exists a point $z_0\in \Omega$ whose positive orbit is dense in $\Omega$ and a non-zero vector $v_0\in T_{z_0}(\mathbb{R}^2)$ with $\Vert d\varphi^n_{z_0}(v_0)\Vert \geq e^{cn}\Vert v_0\Vert$ for any integer $n\geq 0$ and some constant $c>0$.

The conditions (i) and (iii) of Theorem \ref{thm_B} imply that the intersection 
$\tilde A_b=A_b^{(1)}\cap A_b^{(3)}$ is a dense subset of $J_b$ satisfying the following conditions. 

\begin{mcorollary}\label{cor_C}
For any $b\in I\setminus  \{0\}$, there exists a subset $\tilde A_b$ of $J_b$ which has Lebesgue measure positive everywhere in $J_b$ and such that, for any $a\in \tilde A_b$, $\varphi_{a,b}$ does not have an SRB measure supported by the homoclinic set of $p_{a,b}$ but has an SRB measure supported by a strange attractor.
\end{mcorollary}

We finish Introduction by outlining the proofs of Theorems \ref{thm_A} and \ref{thm_B}.
Note that 
a difficulty in our proof is that we need to find desired diffeomorphisms in the \emph{fixed two}-parameter family $\{\varphi_{a,b}\}$ but  \emph{not} a neighborhood of the family in the infinite dimensional space $\mathrm{Diff}^\infty (\mathbb{R}^2)$.
See Section \ref{S_2}. 
Our key mechanism for overcoming it is the \emph{double renormalization} for the two-parameter family.
Most of our effort is devoted to detecting such renormalizations in Section \ref{S_4}.

By using the Implicit Function Theorem, we will define a smooth function $h:I\to \mathbb{R}$ such that, for any $b\in I$, $\varphi_{h(b),b}$ has a homoclinic quadratic tangency $q_{h(b),b}$ near the point $(-2,2)\in \mathbb{R}^2$ unfolding generically with respect to the $a$-parameter family $\{\varphi_{a,b(\mathrm{fixed})}\}$ (Proposition \ref{tangency}).
Then, one can renormalize $\{\varphi_a\}$ with $\varphi_a:=\varphi_{a,b}^w$ in a neighborhood of $q_{h(b),b}$ 
 as in \cite{PT}, 
where $w>0$ is the even integer given in Section \ref{S_2}.
Then, by the Thickness Lemma \cite{N0,PT}, there exists a closed interval $J_b$ such that the one-parameter family $\{\varphi_a\}_{a\in J_b}$ has persistent heteroclinic quadratic tangencies $q_a$ $(a\in J_b)$, and moreover Accompanying Lemma (Lemma \ref{accompany}) implies that all these tangencies unfold generically with respect to $\varphi_a$ (Theorem \ref{persistent tangency}).
By using these results, we also show that $J_b$ contains a dense subset $J'$ such that, for any $\hat a\in J'$,  $\varphi_{\hat a}$ has a homoclinic tangency $q_{\hat a}$ associated to the fixed point $p_{\hat a}:=p_{\hat a,b}$ which also unfolds generically (Theorem \ref{l_0}).
Obviously, Theorems \ref{persistent tangency} and \ref{l_0} imply Theorem \ref{thm_A}.

One can renormalize $\{\varphi_a\}$ again near the tangency $q_{\hat a}$ given 
in Theorem \ref{thm_A}~(ii) for any $\hat a\in J'$.
Applying then standard arguments of Robinson \cite{R} and Newhouse \cite{N2} to our situation, we have an open dense subset $A_b^{(1)}$ of $J_b$ such that $\varphi_a$ has no SRB measure supported by the homoclinic set of $p_{a,b}$ for any $a\in A_b^{(1)}$, and a residual subset $A_b^{(2)}$ of $J_b$ with $A_b^{(2)}\subset A_b^{(1)}$ such that $\varphi_a$ has infinitely many sinks for any $a\in A_b^{(2)}$.
Moreover, applying the results of Wang-Young \cite{WY} to the renormalized maps, we have a dense subset $A_b^{(3)}$ of $J_b$ with Lebesgue measure positive everywhere and such that $\varphi_a$ has a strange attractor supporting an SRB measure if $a\in A_b^{(3)}$.
These results prove Theorem \ref{thm_B}.

\section{Preliminaries}\label{S_1}

First of all, we will review briefly some notations and definitions needed in later sections.
Throughout this section, we suppose that $\{\psi_t\}_{t\in J}$ is a one-parameter family in $\mathrm{Diff}^r(\mathbb{R}^2)$ with $r\geq 3$ such that the parameter space $J$ is an interval.
A family $\{A_t\}_{t\in J}$ of $\psi_t$-invariant subsets of $\mathbb{R}^2$ is called a $t$-\emph{continuation} (or shortly
\emph{continuation}) if, for any $t\in J$ and some $t_0\in J$, there exist homeomorphisms 
$h_t:A_{t_0}\to A_t$ depending on $t$ continuously  such that $h_{t_0}$ is the identity of $A_{t_0}$ and 
$h_t\circ \psi_{t_0}|_{A_{t_0}}=\psi_t|_{A_t}\circ h_t$.

\subsection{Thickness of Cantor sets}

We recall the definition of the thickness given in Newhouse \cite{N1} and Palis-Takens \cite{PT} for a Cantor set $K$ contained in an interval $I$.
A \textit{gap} of $K$ is a connected component of $I\setminus K$ which does not contain a boundary point of $I$.
Let $G$ be a gap and $p$ a boundary point of $G$.
A closed interval $B\subset  I$ is called the \textit{bridge} at $p$ if $B$ is the maximal interval with $G\cap B=\{p\}$ such that $B$ does not intersect any gap whose length is at least that of $G$. 
The  \textit{thickness}  of  $K$ at $p$ is defined by  $\tau(K,p)= \mathrm{Length}(B)/\mathrm{Length}(G)$.
The \textit{thickness} $\tau(K)$ of $K$ is the infimum  over these $\tau(K,p)$ for all boundary points $p$ of gaps of $K$.
Let $K_1,K_2$ be two Cantor sets in $I$ with thickness $\tau_1$ and $\tau_2$ respectively.
Then, Gap Lemma (see \cite{N1,PT}) shows that, if $\tau_1\cdot \tau_2>1$, then either $K_1$ is contained in a gap of $K_2$, or $K_2$ is contained in a gap of $K_1$, or $K_1\cap K_2\neq \emptyset$.
The thickness of a Cantor subset $K$ of a $C^1$ curve $\alpha$ in $\mathbb{R}^2$ is defined similarly by supposing that $\alpha$ is parametrized by arc-length.

\subsection{Persistent quadratic tangencies}\label{s_persistent}

A \textit{basic set} of $\psi_t$ is a non-trivial compact transitive hyperbolic invariant set of $\psi_t$ with a dense subset of periodic orbits.
Suppose that there exist continuations $\{\Lambda_{1,t}\}_{t\in J}$, $\{\Lambda_{2,t}\}_{t\in J}$ of basic sets or saddle fixed points of $\psi_t$ such that $W^s(\Lambda_{1,t_0})$ and $W^u(\Lambda_{2,t_0})$ have a {\it quadratic tangency} $q_{t_0}$ for a $t_0\in J$.
That is, one can choose a coordinate $(x,y)$ on a neighborhood $O$ of $q_{t_0}$ with $q_{t_0}=(0,0)$ and such that
$$L^s_{t_0}\cap O=\{(x,y):\, y=0\}\quad\mbox{and}\quad L^u_{t_0}\cap O=\{(x,y):\, y=ax^2\}$$
for some constant $a\neq 0$, where $L^s_t,L^u_t$ are arcs in $W^s(\Lambda_{1,t})$ and $W^u(\Lambda_{2,t})$ respectively 
which depend on $t$ continuously.
The tangency is called {\it homoclinic} (resp.\ {\it heteroclinic}) if $\Lambda_{1,t_0}=\Lambda_{2,t_0}$ (resp.\ $\Lambda_{1,t_0}\neq \Lambda_{2,t_0}$).
One can choose these $L^s_t,L^u_t$ so that they vary $C^{r-2}$ with respect to $t$, for example see 
Pollicott \cite[Propositions 1 and 2]{Pol} and references therein.
Notice that the assumption $r\geq 3$ continues to be used.

\begin{definition}\label{d_unfold}
The quadratic tangency $q_{t_0}$ {\it unfolds generically} with respect to $\{\psi_t\}_{t\in J}$ if there exist local coordinates on $O$ $C^3$ depending on $t$ and a $C^1$ function $b$ on $J$ satisfying the following conditions.
\begin{itemize}
\item
$L^s_t\cap O$ is given by $y=0$ and $L^u_t\cap O$ by $y=ax^2+b(t)$ for any $t$ near $t_0$.
\item
$b(t_0)=0$ and $\displaystyle \frac{db}{dt}(t_0)\neq 0$.
\end{itemize}
\end{definition}

The family $\{\psi_t\}_{t\in J}$ is said to have \textit{persistent quadratic tangencies unfolding generically} if, for any $t_1\in J$, $W^s(\Lambda_{1,t_1})$ and $W^u(\Lambda_{2,t_1})$ have a quadratic tangency $q_{t_1}$ unfolding generically with respect to $\{\psi_t\}$.

\subsection{Compatible foliations}

Let $\mathcal{F}$ be a foliation consisting of smooth curves in the plane.
A smooth curve $\sigma$ in the plane is said to {\it cross} $\mathcal{F}$ {\it exactly} if each leaf of $\mathcal{F}$ intersects $\sigma$ transversely in a single point and any point of $\sigma$ is passed through by a leaf of $\mathcal{F}$.

Suppose that $\{\Lambda_t\}$ is a continuation of non-trivial basic sets of $\psi_t$ and $\{p_t\}$ is a continuation of saddle fixed points in $\Lambda_t$.
Let $\{I_t\}$ be a set of curves in $W^s_{\mathrm{loc}}(p_t)$ which are shortest among curves in $W^s_{\mathrm{loc}}(p_t)$ containing $\Lambda_t\cap W^s_{\mathrm{loc}}(p_t)$ and depends on $t$ continuously.
According to Lemma 4.1 in Kan-Ko\c{c}ak-Yorke \cite{KKY} based on results in Franks \cite{Fr}, there exists a $t$-parameter 
family of foliations $\mathcal{F}^u_t$ in $\mathbb{R}^2$ satisfying the following conditions.
Such foliations are said to be {\it compatible with} $W_{\mathrm{loc}}^u(\Lambda_t)$.
\begin{enumerate}[\rm (i)]
\item
Each leaf of $W^u_{\mathrm{loc}}(\Lambda_t)$ is a leaf of $\mathcal{F}^u_t$.
\item
$I_t$ crosses $\mathcal{F}^u_t$ exactly.
\item
Leaves of $\mathcal{F}^u_t$ are $C^3$ curves such that themselves, their directions, and their curvatures vary $C^1$ with respect to any transverse direction and $t$.
\end{enumerate}
Similarly, there exist foliations $\mathcal{F}^s_t$ {\it compatible with} $W^s_{\mathrm{loc}}(\Lambda_t)$.
A leaf of $\mathcal{F}^{u/s}_t$ is said to be a \textit{$\Lambda_t$-leaf} if the leaf is contained in $W^{u/s}(\Lambda_t)$.

\subsection{Accompanying Lemma}\label{s_accompany}

We still work with the notation and situation as in the previous subsections.
Accompanying Lemma given below is used to show that some quadratic tangencies $q_t$ unfold generically.

Suppose that there exists a continuation of saddle fixed points $\hat p_t$ of $\psi_t$ other than $p_t$ such that $W^s(\hat p_t)\setminus \{\hat p_t\}$ has a subarc crossing $\mathcal{F}_t^{u(k_0)}:=\psi_t^{k_0}(\mathcal{F}^u_t)$ exactly for some integer $k_0\geq 0$.
Let $\sigma$ be an oriented short segment in $\mathbb{R}^2$ meeting $W^u(\hat p_t)\setminus \{\hat p_t\}$ almost orthogonally in a single point of $\mathrm{Int}(\sigma)$, 
which is denoted by $c_{t}$. 
The Inclination Lemma implies that $W^u_{\mathrm{loc}}(\hat p_t)$ is contained in a small neighborhood of $\mathcal{F}^{u(k_0+j)}_t$ for all sufficiently large integer $j>0$, see Fig.\ \ref{fg_1}. In particular, $\sigma$ contains an arc crossing $\mathcal{F}^{u(k_0+j)}_t$ exactly.
\begin{figure}[hbtp]
\centering
\scalebox{0.65}{\includegraphics[clip]{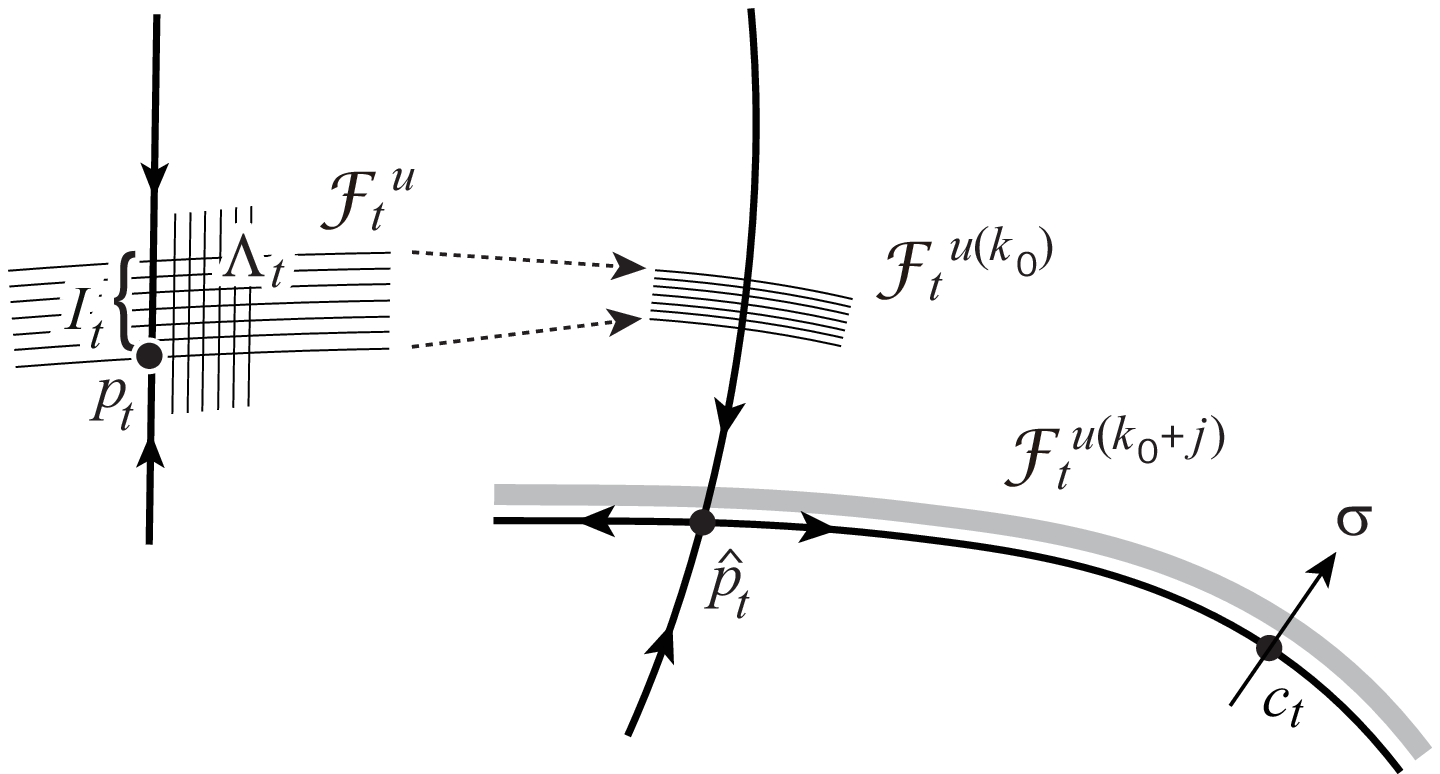}}
\caption{}
\label{fg_1}
\end{figure}

Consider an orientation-preserving arc-length parametrization $\alpha:[v_0,v_1]\to \sigma$ independent of $t$.
Let $v:J\to [v_0,v_1]$ be a $C^1$ function such that $\alpha(v(t))$ is contained in a $\Lambda_t$-leaf of 
$\mathcal{F}_t^{u(k_0+j)}$, 
and let $c:J\to [v_0,v_1]$ be a $C^1$ function satisfying $\alpha(c(t))=c_{t}$.

The following lemma is given in Kiriki-Soma \cite[Lemma 4.1]{KS1} and the proof is in \cite[Appendix A]{KS1}.

\begin{lemma}[Accompanying Lemma]\label{accompany}
For any $\delta>0$ and $t_0\in \mathrm{Int}(J)$, there exists an integer $j_0>0$ and a number $\varepsilon>0$ such that 
any $C^1$ function $v$ as above satisfies
\begin{equation}\label{accomp}
\Bigl\vert \frac{d  v}{d t}(t)-\frac{dc}{dt}(t_0)\Bigr\vert <\delta\end{equation}
if $j\geq j_0$ and $|t-t_0|<\varepsilon$.
\end{lemma}

\section{Continuations of homoclinic tangencies}\label{S_2}
As we have stated in the end of Introduction, arguments in \cite[Section 6.3]{PT} which were 
used to detect diffeomorphisms admitting homoclinic or heteroclinic tangencies in a small neighborhood of 
H\'enon maps in $\mathrm{Diff}^\infty (\mathbb{R}^2)$ can not be applied directly to the 
\emph{fixed} two-parameter family  of original H\'enon maps:
$$
\varphi_{a,b}(x,y)=(y, a-bx+y^2).
$$
In this section, we will  consider a $C^\infty$ function $h:(-\varepsilon,\varepsilon)\to \mathbb{R}$ such that the $b$-parameter family $\{\varphi_{h(b),b}\}$ admits a $b$-continuation of homoclinic quadratic tangencies $q_{h(b),b}$ each of which unfolds generically with respect to the $a$-parameter family $\{\varphi_{a,b(\mathrm{fixed})}\}$. 
Finally, we will obtain essential results for the $a$-parameter family. 

For any element $(a,b)$ of a small neighborhood of $(-2,0)$ in the parameter space, $\varphi_{a,b}$ has the two fixed points $p_{a,b}^\pm$ with
\begin{equation}\label{p_ab}
p_{a,b}^\pm=(y_{a,b}^\pm,y_{a,b}^\pm),\quad\mbox{where}\quad
y_{a,b}^\pm=\frac{1+b\pm\sqrt{(1+b)^2-4a}}{2}.
\end{equation}
For short, we set $p_{a,b}^+=p_{a,b}$ and $y_{a,b}^+=y_{a,b}$.
Then, the eigenvalues of the differential $(D\varphi_{a,b})_{p_{a,b}}$ at $p_{a,b}$ are
\begin{equation}\label{e_v}
\lambda_{a,b}={y_{a,b}-\sqrt{y_{a,b}^2-b}},\ \sigma_{a,b}={y_{a,b}+\sqrt{y_{a,b}^2-b}}.
\end{equation}
Thus, for any $(a,b)\approx (-2,0)$ with $b\neq 0$, the eigenvalues satisfy
\begin{equation}\label{lambda_sigma}
0<\vert\lambda_{a,b}\vert <1<\sigma_{a,b}\quad\mbox{and}\quad \vert \lambda_{a,b}\vert \sigma_{a,b}<1.
\end{equation}
A fixed point satisfying the condition (\ref{lambda_sigma}) is called a {\it dissipative saddle} fixed point.

When $b=0$, $\varphi_{a,0}$ is not a diffeomorphism.
Even in this case, one can define the stable and unstable manifolds associated to $p_{a,0}$ in a usual manner.
The stable manifold $W^s(p_{a,0})$ of $\varphi_{a,0}$ is the horizontal line $y=y_{a,0}$ in $\mathbb{R}^2$ passing through $p_{a,0}$.
Hence, $W^s(p_{a,0})$ contains the horizontal segment $S_{a,0}=\{(x,y_{a,0});\vert x\vert\leq 5/2\}$.
By the Stable Manifold Theorem (see e.g.\ Robinson \cite[Chapter 5, Theorem 10.1]{Rb}), for any $(a,b)\approx (-2,0)$,
there exists an almost horizontal segment $S_{a,b}\subset W^s(p_{a,b})$ containing $p_{a,b}$ which $C^\infty$ depends on $(a,b)$ and such that one of the end point of $S_{a,b}$ is in the vertical line $x=-5/2$ and the other in $x=5/2$.
In particular, each $S_{a,b}$ is represented as the graph of a $C^\infty$ function $\eta_{a,b}$ of $x$ $C^\infty$ depending on $(a,b)$, that is,
$$S_{a,b}=\{(x,\eta_{a,b}(x));\vert x\vert \leq 5/2\}.$$
Since the family $\{\eta_{a,b}\}$ $C^\infty$ converges to the constant function $\eta_{a_0,0}$ uniformly as $(a,b)\rightarrow (a_0,0)$,
\begin{equation}\label{to_zero}
\begin{split}
&\lim_{(a,b)\rightarrow (a_0,0)}\max\left\{\left\vert\frac{d\eta_{a,b}}{d x}(x)\right\vert\,;\,-5/2\leq x\leq 5/2\right\}=0,\\
&\lim_{(a,b)\rightarrow (a_0,0)}\max\left\{\left\vert\frac{d^2\eta_{a,b}}{d x^2}(x)\right\vert\,;\,-5/2\leq x\leq 5/2\right\}=0.
\end{split}
\end{equation}

From the definition, the unstable manifold $W^u(p_{a,0})$ consists of the points $q\in \mathbb{R}^2$ which admits a sequence $\{q_n\}_{n=0}^\infty$ in $\mathbb{R}^2$ with $q_0=q$, $q_n\in \varphi_{a,0}^{-1}(q_{n-1})$ for $n=1,2,\dots$ and $\lim_{n\rightarrow \infty}q_n=p_{a,0}$.
In particular, $W^u(p_{a,0})$ is contained in the parabolic curve $\mathrm{Im}(\varphi_{a,0})=\{(x,x^2+a); -\infty <x<\infty\}$.
Then, it is not hard to show that
$$W^u(p_{a,0})=\{(x,x^2+a); a \leq x<\infty\}$$
for any $a\approx -2$.
Again by the Stable Manifold Theorem, for any $(a,b)\approx (-2,0)$ (possibly $b=0$), there exist short curves $T_{a,b}$ in $W^u_{\mathrm{loc}}(p_{a,b})$ with $\mathrm{Int}(T_{a,b})\ni p_{a,b}$ and varying $C^\infty$ with respect to $(a,b)$.
Thus, for any integer $m>0$, $T^{(m)}_{a,b}=\varphi_{a,b}^m(T_{a,b})$ $C^\infty$ converges to $T^{(m)}_{a_0,0}=\varphi_{a_0,0}^m(T_{a_0,0})$ as $(a,b)\rightarrow (a_0,0)$.
Intuitively, the curve $T^{(m)}_{a_0,0}$ is obtained by `folding' $T^{(m)}_{a,b}$ when $m$ is large enough, see Fig.\ \ref{fg_2}.
\begin{figure}[hbt]
\begin{center}
\scalebox{0.85}{\includegraphics[clip]{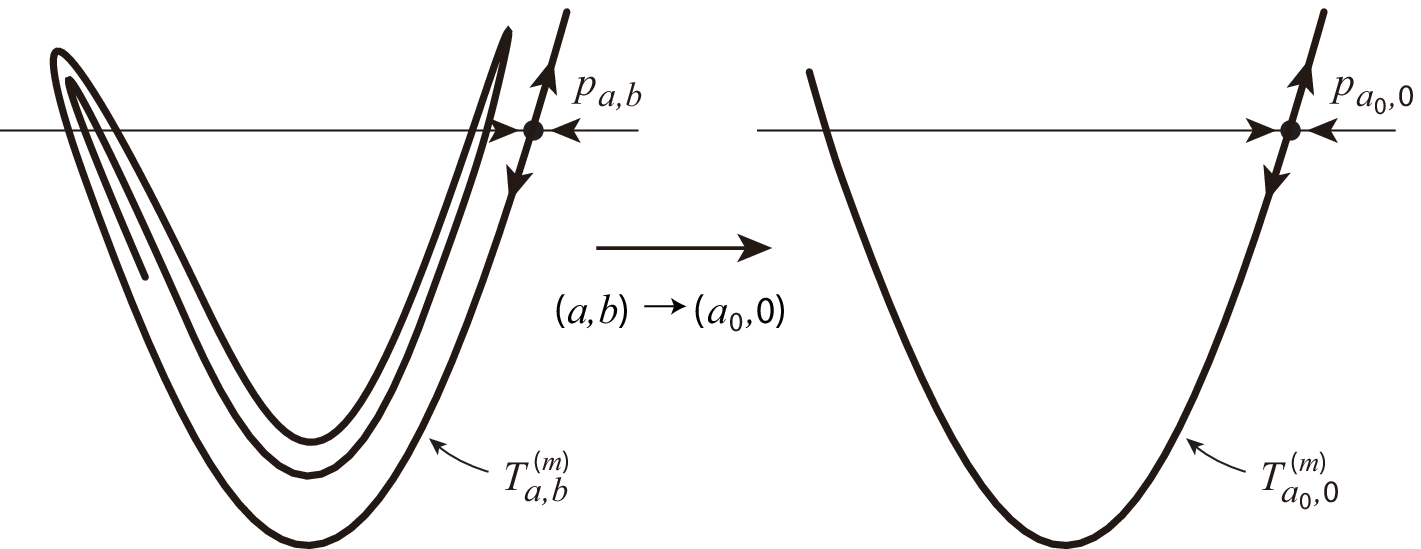}}
\caption{}
\label{fg_2}
\end{center}
\end{figure}
Let $a_0$ be any number sufficiently close to $-2$.
For any $(a,b)\approx (a_0,0)$, take a curve $U_{a,b}$ in $W^u(p_{a,b})$ with $p_{a,b}$ as one of its end points  and $C^\infty$ converging to $U_{a_0,0}=\{(x,x^2+a_0);a_0\leq x\leq y_{a_0,0}\}\subset W^u(p_{a_0,0})$ injectively as $(a,b)\rightarrow (a_0,0)$.
Set
$$U^{1}_{a,b}=\varphi_{a,b}(U_{a,b})\setminus U_{a,b}, \quad 
U^{2}_{a,b}=\varphi_{a,b}(U^{1}_{a,b})\quad\mbox{and}\quad U^{3}_{a,b}=\varphi_{a,b}(U^{2}_{a,b}).
$$
Fix a sufficiently small $\delta >0$, let $l_{a,b}$ be the curve in $U_{a,b}^1$ such that its end points are contained in the vertical lines 
$x=\pm \delta$.
The curve $l_{a,b}$ is represented by the graph of a $C^\infty$ function $y=\zeta_{a,b}(t)$ $(-\delta\leq t\leq \delta)$ $C^\infty$ depending on $(a,b)$ and such that $\{\zeta_{a,b}\}$ uniformly $C^\infty$ converges to the function $\zeta_{a,0}$ with $\zeta_{a,0}(t)=t^2+a$ as $b\rightarrow 0$.
Note that $l_{a,b}^1=\varphi_{a,b}(l_{a,b})$ has a maximal point of $U^2_{a,b}$ contained in a small neighborhood $V(-2,2)$ of $(-2,2)$ in $\mathbb{R}^2$, see Fig.\ \ref{fg_3}.
\begin{figure}[hbt]
\begin{center}
\scalebox{0.85}{\includegraphics[clip]{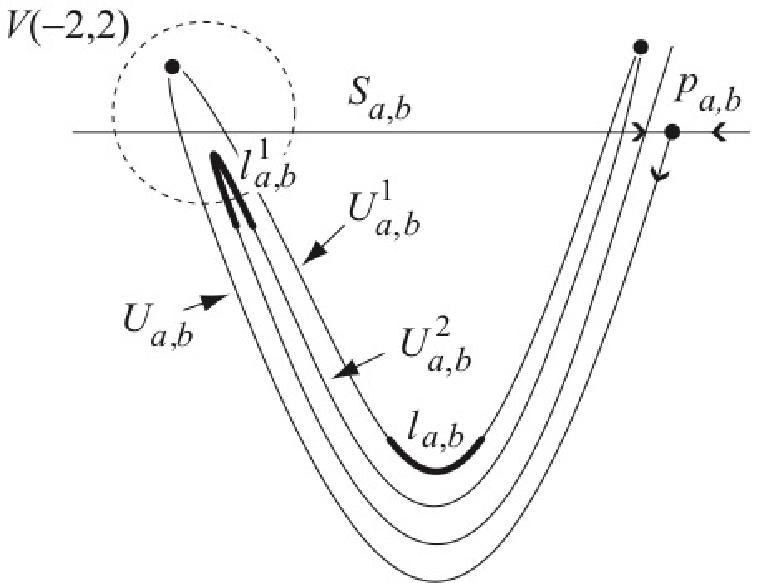}}
\caption{The case of $b>0$.}
\label{fg_3}
\end{center}
\end{figure}

Consider $C^\infty$ diffeomorphisms $\Psi_{a,b}$ on $\mathbb{R}^2$ with $\Psi_{a,b}(x,y)=(x,y-\eta_{a,b}(x))$ if $\vert x\vert \leq 5/2$.
Then, for any $(a,b)\approx (-2,0)$, the image $\Psi_{a,b}(S_{a,b})$ is equal to the horizontal segment $[-5/2,5/2]\times \{0\}$ in $\mathbb{R}^2$.
In particular, any quadratic tangency of $\Psi_{a,b}(S_{a,b})$ and a curve $l$ in $\mathbb{R}^2$ is either a maximal or minimal point of $l$.
Let $\theta_{a,b}:[-\delta,\delta]\to \mathbb{R}$ be the $C^\infty$ function such that $\theta_{a,b}(t)$ is the $y$-entry of the coordinate of $\Psi_{a,b}\circ \varphi_{a,b}(t,\zeta_{a,b}(t))\in \mathbb{R}^2$, that is,
$$\theta_{a,b}(t)=a-bt+\zeta_{a,b}(t)^2-\eta_{a,b}(\zeta_{a,b}(t)).$$
Note that $\theta_{a,b}$ $C^\infty$ depends on $(a,b)$.
By (\ref{to_zero}), both $\vert d\eta_{a,b}(x)/dx\vert$, $\vert d^2\eta_{a,b}(x)/dx^2\vert$ $(-5/2\leq  x \leq 5/2)$ are sufficiently small.
Since moreover $d\zeta_{a,b}(t)/dt\rightarrow 2t$ and $d^2\zeta_{a,b}/dt^2(t)\rightarrow 2$ as $b\rightarrow 0$,
$d^2\theta_{a,b}(t)/dt^2\approx -8\neq 0$ for any $(a,b)\approx (-2,0)$ and $t\approx 0$.
Hence, there exists a unique point $t_{a,b}\in (-\delta,\delta)$ at which $\theta_{a,b}$ has the maximal value $\theta_{a,b}(t_{a,b})$.
Here, we set
$$H(a,b):=\theta_{a,b}(t_{a,b})\quad\mbox{and}\quad q_{a,b}:=\varphi_{a,b}(t_{a,b},\zeta_{a,b}(t_{a,b}))\in l_{a,b}^1.$$
The point $q_{a,b}$ is a candidate of the quadratic tangency of $S_{a,b}$ and $W^u(p_{a,b})$ for suitable pairs of $(a,b)$.

From the definition, we know that $H$ is a $C^\infty$ function of $(a,b)\approx (-2,0)$.
Since $S_{a,0}$ is in the horizontal line $y=y_{a,0}$,
$$H(a,0)=a^2+a-y_{a,0}=a^2+a-\frac{1+\sqrt{1-4a}}{2}.$$
Thus, we have
$$
H(-2,0)=0,\quad \frac{\partial H}{\partial a} (a,b) \approx -\frac{8}{3}\neq 0
$$
for any $(a,b)\approx (-2,0)$.
By the Implicit Function Theorem, there exists a $C^\infty$ function $h:I_\varepsilon=(-\varepsilon, \varepsilon)\to \mathbb{R}$ for a small $\varepsilon>0$ satisfying
\begin{equation}\label{implicit function}
h(0)=-2,\quad H(h(b),b)=0,\quad  \frac{dh}{d b} =
-\frac{{\partial H}/{\partial b}}{{\partial H}/{\partial a}}.
\end{equation}
For each $b\in I_\varepsilon$, since
$$\theta_{h(b),b}(t_{h(b),b})=0,\quad \frac{d\theta_{h(b),b}}{dt}(t_{h(b),b})=0\quad\mbox{and}\quad \frac{d^2\theta_{h(b),b}}
{dt^2}(t_{h(b),b})\approx -8\neq 0,$$ 
$q_{h(b),b}$ is a quadratic tangency of $W^u(p_{h(b),b})$ and $S_{h(b),b}\subset W^s(p_{h(b),b})$.

\begin{remark}\label{r_1}
We note that $S_{h(b),b}$ is almost horizontal, but not in general strictly horizontal when $b\neq 0$.
Thus, the slope of $S_{h(b),b}$ at the tangency $q_{h(b),b}$ is not necessarily zero, and hence $q_{h(b),b}$ may not be a maximal point of $l_{h(b),b}^1$.
On the other hand, $\Psi_{h(b),b}(q_{h(b),b})$ is a unique maximal point of $\Psi_{h(b),b}(l_{h(b),b}^1)$ tangent to 
the horizontal segment $\Psi_{h(b),b}(S_{h(b),b})$.
This is a technical reason for introducing the coordinate change by $\Psi_{a,b}$ to define the function $H$.
\end{remark}

With the notation as above, we will prove the following proposition.

\begin{proposition}\label{tangency}
If $\varepsilon>0$ is sufficiently small, then for any $b\in  I_{\varepsilon}$, $W^u(p_{h(b),b})$ and $W^s(p_{h(b),b})$ have a transverse intersection and the quadratic tangency $q_{h(b),b}$ unfolds generically with respect to the $a$-parameter family $\{\varphi_{a,b(\mathrm{fixed})}\}$.
\end{proposition}
\begin{proof}
Suppose first that $b$ is any positive number sufficiently close to $0$.
Then, since $l_{h(b),b}^1$ is tangent to $S_{h(b),b}\cap V(-2,2)$ at $q_{h(b),b}$, $U_{h(b),b}\cup U_{h(b),b}^1$ meets $S_{h(b),b}\cap V(-2,2)$ transversely in two points, see Figure \ref{fg_3}.
On the other hand, for any $b<0$ sufficiently close to $0$, $U_{h(b),b}\cup U_{h(b),b}^1$ is disjoint from $S_{h(b),b}\cap V(-2,2)$.
But, in this case, $U^3_{h(b),b}$ meets $S_{h(b),b}\cap V(-2,2)$ transversely in two points.
This shows the former assertion.

From the fact that $\partial H(a,b)/\partial a\vert_{a=h(b)}\neq 0$, we know that the tangency $q_{h(b),b}$ unfolds generically with respect to the $a$-parameter family $\{\varphi_{a,b(\mathrm{fixed})}\}$.
This completes the proof.
\end{proof}

We denote the graph of $h$ in the $ab$-space by $\mathcal{H}$, that is, $\mathcal{H}=\{(h(b),b);\,b\in I_\varepsilon\}$.
For any $(\hat a,\hat b)\in \mathcal{H}$ with $\hat b\neq 0$, one can take a sufficiently thin rectangle $R\subset\mathbb{R}^2$ with $\mathrm{Int}R\supset S_{\hat a,\hat b}$ and an even integer $w>0$ so that 
$R\cap \varphi_{\hat a,\hat b}^w(R)$ has curvilinear rectangle components $R_0,R_1$ such that $\mathrm{Int}R_0$ contains the saddle fixed point $p_{\hat a,\hat b}$, and  $\mathrm{Int}R_1$ contains a homoclinic transverse point associated with $p_{\hat a,\hat b}$ in $V(-2,2)$, but $R_0\cup R_1$ is disjoint from the homoclinic tangency $q_{\hat a,\hat b}$ given in Proposition \ref{tangency}, see Figure \ref{fg_4}.
\begin{figure}[hbt]
\begin{center}
\scalebox{1}{\includegraphics[clip]{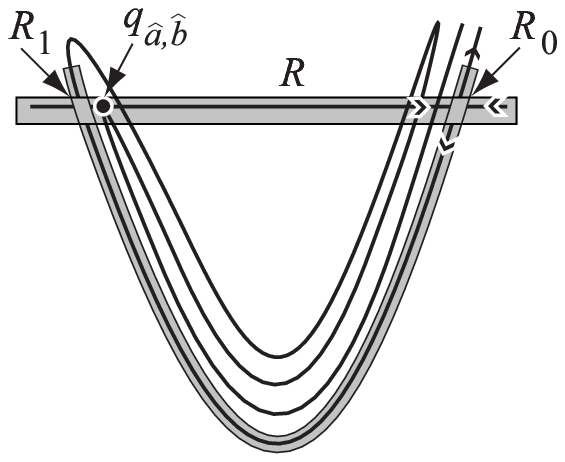}}
\caption{The case of $\hat b>0$.}
\label{fg_4}
\end{center}
\end{figure}
One can take such an $R$ so that the intersection
$${\Lambda_{\hat a,\hat b}^{\mathrm{out}}}:=\bigcap_{n=-\infty}^\infty \varphi_{\hat a,\hat b}^{wn}(R_0\cup R_1)$$ 
is a horseshoe basic set of $\varphi_{\hat a,\hat b}^w$.    
Note that, since $w$ is an even integer, $\varphi_{\hat a,\hat b}^w$ is an orientation-preserving diffeomorphism even when $\hat b<0$.
Here, the superscript ``out'' implicitly suggests that $\Lambda_{\hat a,\hat b}^{\mathrm{out}}$ is 
in the outside of a small neighborhood of the homoclinic tangency $q_{\hat a,\hat b}$.
Then, for any $(a,b)$ sufficiently close to $(\hat a,\hat b)$ (and hence in particular $b\neq 0$), 
there exists a continuation of basic sets $\Lambda_{a,b}^{\mathrm{out}}$ of $\varphi_{a,b}^w$ based at $\Lambda_{\hat a,\hat b}^{\mathrm{out}}$.

\begin{lemma}\label{outer_horseshoe}
For any $(\hat a,\hat b)\in \mathcal{H}$ with $\hat b\neq 0$, 
there exists an open neighborhood $\mathcal{O}=\mathcal{O}(\hat a,\hat b)$ of $(\hat a,\hat b)$ in the $ab$-space and a constant $c=c(\hat a,\hat b)>0$ such that, for any $(a,b)\in \mathcal{O}$, the thickness $\tau(\Lambda_{a,b}^{\mathrm{out}}\cap W^s(p_{a,b}))$ is greater than $c$.
\end{lemma}
\begin{proof}
Since $\Lambda_{\hat a,\hat b}^{\mathrm{out}}\cap W^s(p_{\hat a,\hat b})$ is a dynamically defined Cantor set, $\tau(\Lambda_{\hat a,\hat b}^{\mathrm{out}}\cap W^s(p_{\hat a,\hat b}))$ is positive, see \cite[p.\ 80, Proposition 7]{PT}.
Since moreover $\tau(\Lambda_{a,b}^{\mathrm{out}}\cap W^s(p_{a,b}))$ is continuous on $(a,b)$, see \cite[p.\ 85, Theorem 2]{PT}, one can have $c>0$ satisfying our desired property if the neighborhood $\mathcal{O}$ is taken sufficiently small.
\end{proof}

\section{Double renormalization near quadratic tangencies}\label{S_4}

In this section, we will work with the notations  as in Section \ref{S_2} and 
reform renormalizations near a homoclinic tangency suitable to the proof of our main theorem.

Take an element $b\in I\setminus  \{0\}$ arbitrarily and fix throughout this section.
For simplicity, we set
\begin{equation}\label{varphi}
\varphi_a=\varphi^w_{a,b}\quad\mbox{and}\quad p_a=p_{a,b},
\end{equation}
where $w> 0$ is the even integer given in the paragraph preceding Lemma \ref{outer_horseshoe}.

\begin{lemma}\label{thick horseshoe}
There exists a closed interval $J_{b,1}$  arbitrarily close $h(b)$ 
satisfying the following conditions.
\begin{enumerate}[\rm (i)]
\item
There exist positive integers $n_0, m$ such that, for any $a\in J_{b,1}$  and any $n\geq n_0$, 
$\varphi_a$ has a basic set $\Lambda_{a,n}^m$ containing a  saddle periodic point $Q^m_{a,n}$,  
and besides, 
the thickness  of 
$\Lambda^m_{a,n}\cap W^u(Q^m_{a,n})$ is greater than  $2/c$, where $c=c(h(b),b)$ is the constant given in Lemma \ref{outer_horseshoe}.
\item
For any $a\in J_{b,1}$, the thickness of $\Lambda_{a,b}^{\mathrm{out}}\cap W^s(p_{a,b})$ is greater than $c$.
\end{enumerate}
\end{lemma}
\begin{proof}
(i)
By Proposition \ref{tangency}, $\varphi_{h(b)}$ has the homoclinic quadratic tangency $q_{h(b),b}$ in $V(-2,2)$ associated with the dissipative saddle fixed point $p_{h(b)}$ which unfolds generically with respect to the one-parameter family 
$\{\varphi_a\}_{a\in J_{b,0}}$.
We may suppose 
\begin{itemize}
\item
$p_a=(0,0)$.
\item
When $a=h(b)$, both the points $r=(1,0)$ and $ r^\prime=(0,1)\in U$ belong to the orbit of the homoclinic tangency  $q_{h(b),b}$ and satisfy $\varphi_{h(b)}^N(r^\prime)=r$ for some integer $N>0$.
\end{itemize}

By the work of Romero \cite[Theorem D]{R94} based on \cite{AS}  (see also in  \cite{GS90, MR97}), 
without the hypothesis of smooth linearization as in \cite{PT}, 
we get the following renormalization near the homoclinic tangency  $q_{h(b),b}$: 
 for any sufficiently large integer $n>0$, 
one can obtain a $C^{\infty}$ reparametrization $\Theta_{n}$ on $\mathbb{R}$ and an $a$-dependent $C^2$ coordinate change $\Phi_{n}$ on $\mathbb{R}^2$ satisfying the followings
\begin{itemize}
\item
$d\Theta_n(\bar a)/d\bar a>0$.
\item
$\Theta_n(\bar a)$ (resp.\ $\Phi_n(\bar x,\bar y)$) converges as $n\rightarrow \infty$ locally uniformly to the map with constant value $h(b)\in \mathbb{R}$ (resp.\ $r\in \mathbb{R}^2$).
\item
For any $\bar a\in \mathbb{R}$, the diffeomorphisms $\psi_{\bar a,n}$ on $\mathbb{R}^2$, defined by
$$(\bar{x},\bar{y})\mapsto \psi_{\bar{a}, n}( \bar{x},   \bar{y}):=
  \Phi_{n}^{-1}\circ\varphi_{\Theta_{n}(\bar{a})}^{N+n}\circ \Phi_{n}(\bar{x},\bar{y}),$$
$C^2$ converge as $n\rightarrow \infty$ locally uniformly to the endomorphism $\psi_{\bar a}$ with
$$\psi_{\bar{a}}( \bar{x},   \bar{y})=( \bar{y},   \bar{y}^{2}+\bar{a} ).$$
\end{itemize}

Furthermore, by \cite[p.\ 124, Proposition]{PT}, for each integer $m\geq 3$, we have a small closed interval $\bar{J}$ with $\mathrm{Int} \bar J\ni -2$ and an integer $n_0>0$ satisfying following conditions.
\begin{itemize}
\item For any $\bar{a}\in \bar{J}$ and any integer $n\geq n_0$, 
$\psi_{\bar{a}, n}$ has a saddle fixed point  $P_{\bar{a},n}$ in a small neighborhood of $(2,2)$ 
and a basic set $\Lambda^m_{\bar{a},n}$ containing a saddle  $m$-periodic point $Q^m_{\bar{a},n}$.
\item
$\Theta_n(\bar J)\subset J_{b,0}$.
\item
For any $\bar{a}\in \bar{J}$ and any integer $n\geq n_0$, the thickness $\tau(\Lambda^m_{\bar{a},n}\cap W^u(Q^m_{\bar{a},n}))$ of the Cantor set $\Lambda^m_{\bar{a},n}\cap W^u(Q^m_{\bar{a},n})$ is greater than an arbitrarily given constant if $m$ is sufficiently large.
\end{itemize}
So, one can take such $\bar{J}$, $n_0$ and $m$  so that 
$$
\tau(\Lambda^m_{\bar{a},n}\cap W^u(Q^m_{\bar{a},n}))\geq \frac{2}{c}
$$
 for any $\bar{a}\in \bar{J}$
 and  $n\geq n_0$. 
Then, the proof of (i) is completed by letting $J_{b,1}=\Theta_n(\bar{J})$, $\Lambda^m_{a,n}=\Phi_n(\Lambda^m_{\bar a,n})$ and $Q^m_{a,n}=\Phi_n(Q^m_{\bar a,n})$.

(ii)
For the proof, it suffices to retake the integer $n_0$ in (i) so that $J_{b,1}=\Theta_n(\bar{J})$ satisfies $J_{b,1}\times \{b\}\in \mathcal{O}(h(b),b)$ for any $n\geq n_0$, where $\mathcal{O}(h(b),b)$ is the open subset of the $ab$-space given in Lemma \ref{outer_horseshoe}.
\end{proof}

\begin{convent}\label{conv_1}
From now on, we will suppose that $x,y,a$ and $\bar x,\bar y,\bar a$ are related by $a=\Theta_n(\bar a)$ and $(x,y)=\Phi_n(\bar x,\bar y)$ whenever the $n$ is selected.
Moreover, for any parametrized subset $Y_{\bar{a}}$ of the $\bar{x}\bar{y}$-space (resp.\ $Z_a$ of the $xy$-space), we will denote the image $\Phi_n(Y_{\bar{a}})$ (resp.\  $\Phi_n^{-1}(Z_a)$) in the $xy$-space (resp.\ the $\bar x\bar y$-space) again by  $Y_{\bar{a}}$ (resp.\ $Z_a$).
\end{convent}

For any sufficiently large integer $n>0$, one can define a foliation $\mathcal{F}^s_{\bar{a},n}$ in the $\bar{x}\bar{y}$-space compatible with $W^s_{\mathrm{loc}}(\Lambda^m_{\bar{a},n})$ such that there exists an arc $l^u_{\bar{a},n}$ in $W^u(P_{\bar{a},n})$ crossing $\mathcal{F}^s_{\bar{a},n}$ exactly as shown in Figure \ref{fg_6}.
Let $S_{\bar{a},n}$ be the curve in $W^s(P_{\bar{a},n})$ containing $P_{\bar a,n}$ and such that one of the end point of $S_{a,b}$ is in the vertical line $\bar x=-5/2$ and the other in $\bar x=5/2$. 
For any sufficiently large integer $i>0$, by the Inclination Lemma (see for example \cite[Chapter 5, Theorem 11.1]{Rb}), 
one can have foliations $\mathcal{F}^{s(i)}_{\bar{a},n}$ obtained by shortening the leaves of $\psi^{-i}_{\bar{a},n}(\mathcal{F}^s_{\bar{a},n})$ so that all leaves of $\mathcal{F}^{s(i)}_{\bar{a},n}$ are well approximated by $S_{\bar{a},n}$, see in Figure \ref{fg_6}.
\begin{figure}[hbt]
\begin{center}
\scalebox{0.8}{\includegraphics[clip]{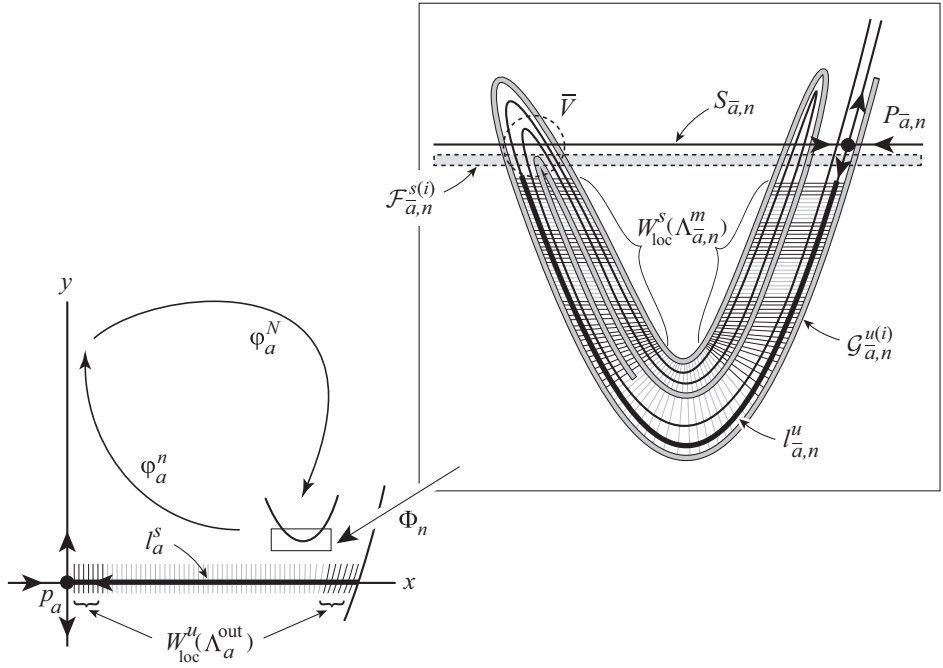}}
\caption{In the left (resp.\ right) side figure, the short segments meeting $l^s_a$ (resp.\ $l^u_{\bar a,n}$) transversely are leaves of $\mathcal{G}_a^u$ (resp.\ $\mathcal{F}^s_{\bar a,n}$).}
\label{fg_6}
\end{center}
\end{figure}

Since $(a,b)\in \mathcal{O}$ for any $a\in J_{b,1}$, we have the basic set
 $\Lambda_{a}^{\mathrm{out}}:=\Lambda_{a,b}^{\mathrm{out}}$ given in Lemma \ref{outer_horseshoe}.
Let $\mathcal{G}^u_{a}$ be a foliation compatible with $W^u_{\mathrm{loc}}(\Lambda_{a}^{\mathrm{out}})$ and 
$l^s_{a}$ the segment in $W^s(p_a)$ crossing $\mathcal{G}^u_{a}$ exactly as shown in Figure \ref{fg_6}.
By \cite[p.\ 125, Proposition 1]{PT}, there exists a compact arc $\sigma^s_{\bar{a},n}$ in $W^s(P_{\bar{a},n})$ 
containing $P_{\bar{a},n}$ and converges in the $xy$-space to an arc in $W^s(p_{a})$ which contains at least one fundamental domain as $n\rightarrow \infty$.
Moreover, Remark 1 in \cite[p.\ 129]{PT} shows that, for all sufficiently large $n$ and some $j>0$, $\varphi^{-(n+j)}_a(\sigma^s_{\bar a,n})$ meets $\mathcal{G}_a^u$ non-trivially and transversely.
From this fact together with the Inclination Lemma, for any sufficiently large integer $i>0$,  
one can have a foliation $\mathcal{G}^{u(i)}_{\bar{a},n}$ in the $\bar x\bar y$-space as in Figure \ref{fg_6} obtained by shortening the leaves of 
the $\Phi_n^{-1}$-image of 
$\varphi^{i}_{a}(\mathcal{G}^u_{a})$ 
so that all leaves  of $\mathcal{G}^{u(i)}_{\bar{a},n}$ are well approximated by a single curve $L^{u}_{\bar a,n}$ in $W^u(P_{\bar{a},n})$ passing through $P_{\bar{a},n}$ and with three maximal points as illustrated in Fig.\ \ref{fg_7}.
\begin{figure}[hbt]
\begin{center}
\scalebox{0.8}{\includegraphics[clip]{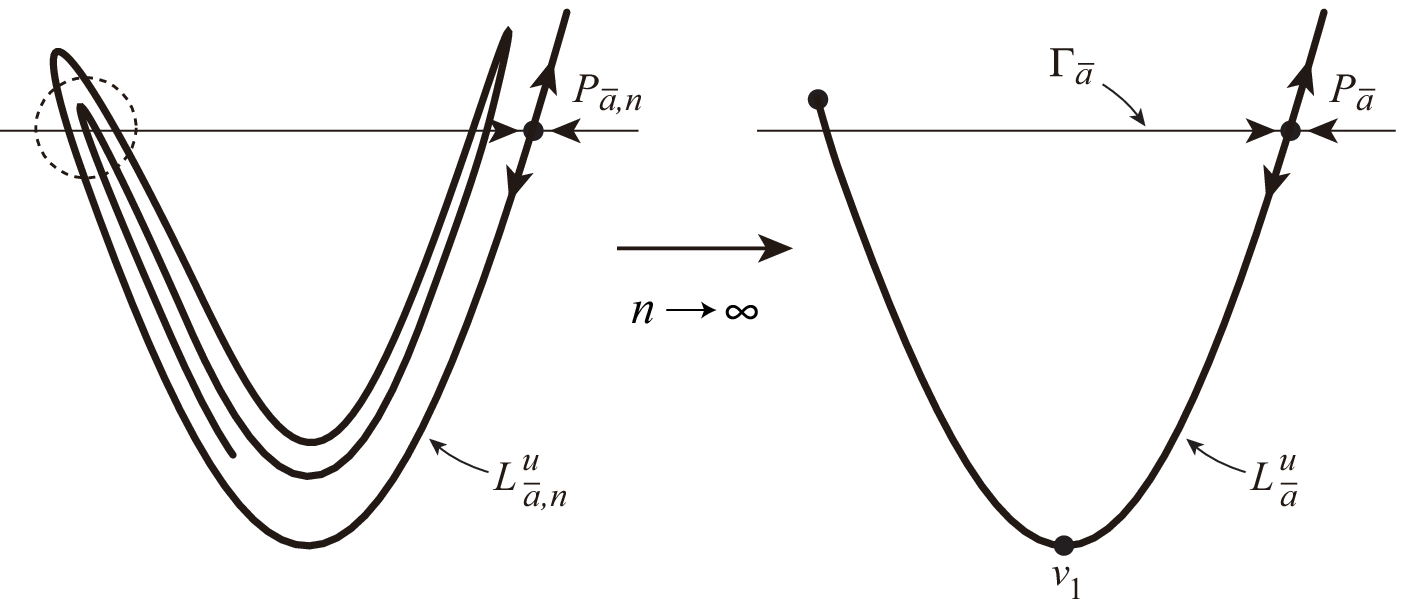}}
\caption{}
\label{fg_7}
\end{center}
\end{figure}
Here, we take the sequence $\{L^u_{\bar a,n}\}_n$ so that it $C^2$ converges onto the graph $L_{\bar a}^u$ of $\bar y=\bar x^2+\bar a$ with $\bar a\leq \bar x\leq \bar x_1(\bar a)+\alpha$ for some $\alpha>0$ as $n \rightarrow \infty$.

Recall that a leaf of $\mathcal{F}^{s(i)}_{\bar{a},n}$ (resp.\ $\mathcal{G}^{u(i)}_{\bar{a},n}$) is said to be a $\Lambda^m_{\bar{a},n}$-leaf (resp.\ $\Lambda^{\mathrm{out}}_{a}$-leaf) if the leaf is contained in $W^s(\Lambda^m_{\bar{a},n})$ (resp.\ $W^u(\Lambda^{\mathrm{out}}_{a})$).
Consider  $\Lambda^m_{\bar{a},n}$-leaves $f_{\bar a,n}^{(i)}$ of $\mathcal{F}^{s(i)}_{\bar{a},n}$ and $\Lambda^{\mathrm{out}}_{a}$-leaves $g_{\bar a,n}^{(i)}$ of $\mathcal{G}^{u(i)}_{\bar{a},n}$ $C^1$ depending on $\bar a\in \bar J_{\bar \delta}$, where $\bar\delta>0$ is 
taken sufficiently small so that $\bar J_{\bar \delta}=[-2-\bar\delta,-2+\bar\delta]$ is contained in the interval $\bar J$ given in the proof of Lemma \ref{thick horseshoe}.
Let $\bar{V}$ be a small open neighborhood of $(-2,2)$ in the $\bar{x}\bar{y}$-space and $\xi$ a vertical segment in $\bar{V}$ meeting $f_{\bar a,n}^{(i)}$, $g_{\bar a,n}^{(i)}$, $S_{\bar a,n}$, $L^u_{\bar a,n}$ almost orthogonally for any $\bar a\in \bar J_{\bar \delta}$, see Figure \ref{fg_8}. 
\begin{figure}[hbt]
\begin{center}
\scalebox{0.85}{\includegraphics[clip]{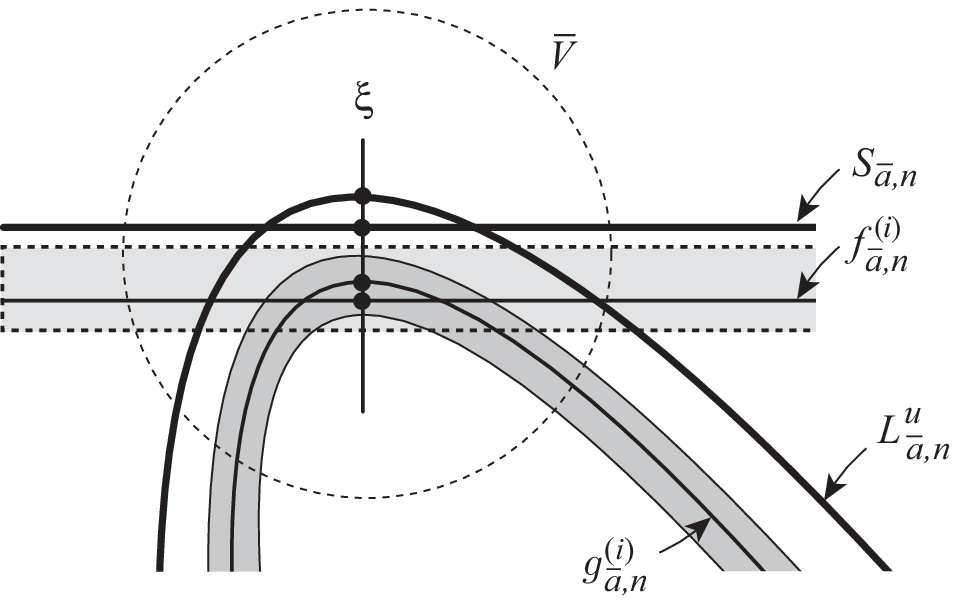}}
\caption{}
\label{fg_8}
\end{center}
\end{figure}
We denote the the $\bar y$-entries of the coordinates of the intersections $f_{\bar a,n}^{(i)}\cap \xi$, $g_{\bar a,n}^{(i)}\cap \xi$, $S_{\bar a,n}\cap\xi$, $L^u_{\bar a,n}\cap\xi$ by $\bar y(f_{\bar a,n}^{(i)})$, $\bar y(g_{\bar a,n}^{(i)})$, $\bar y(S_{\bar a,n})$ and $\bar y(L^u_{\bar a,n})$ respectively.

\begin{lemma}\label{velocity}
With the notation as above, if we take $\bar \delta>0$, $\bar V$ sufficiently small and  $n$ sufficiently large, then there exists an integer $i_0=i_0(n)>0$ such that, if $i\geq i_0$, the derivatives of $\bar y(f_{\bar a,n}^{(i)})$ and $\bar y(g_{\bar a,n}^{(i)})$ satisfy
\begin{equation}\label{ineq}
\frac{d \bar y(f_{\bar a,n}^{(i)})}{d \bar a}-\frac{d \bar y(g_{\bar a,n}^{(i)})}{d \bar a}>2.
\end{equation}
\end{lemma}

\begin{proof}
For any $\bar a\in \bar J_{\bar \delta}$, consider the endomorphism $\psi_{\bar a}(\bar x,\bar y)=(\bar y,\bar y^2+\bar a)$ given in the proof of Lemma \ref{thick horseshoe}.
The point $P_{\bar a}=(\bar x_1,\bar y_1)$ with $\bar x_1(\bar a)=\bar y_1(\bar a)=(1+\sqrt{1-4\bar{a}}\,)/2$ is a fixed point of $\psi_{\bar a}$ close to $(2,2)$.
If necessary slightly extending $\xi$, we may assume that $\xi$ meets the line $\Gamma_{\bar a}: \bar y=(1+\sqrt{1-4\bar{a}}\,)/2$ orthogonally in a single point.
Since $d \bar y_1/d \bar a\rightarrow -1/3$ as $\bar a\rightarrow -2$ and $S_{\bar a,n}$ $C^2$ converges to a segment in $\Gamma_{\bar a}$ as $n\rightarrow \infty$, $d \bar y(S_{\bar a,n})/d \bar a$ is well approximated by $-1/3$ for all sufficiently large $n$.

The curve $L_{\bar a}^u$ has the point $v_1=(0,\bar a)$ as a unique minimal point.
The $\psi_{\bar{a}}$-image $(\bar x_2,\bar y_2)=(\bar a,\bar a^2+\bar a)$ of $v_1$ is the left side end point of $L^u_{\bar a}$.
Since $d \bar y_2/d \bar a=2\bar a+1\rightarrow -3$ as $\bar a\rightarrow -2$, $d \bar y(L_{\bar a,n}^u)/d \bar a$ is well approximated by $-3$ for all sufficiently large $n$ and any $\bar a\in \bar J_{\bar \delta}$ if we take $\bar\delta>0$ sufficiently small.

Now, we apply Accompanying Lemma (Lemma \ref{accompany}) to the ordered families $(\Lambda_a^{\mathrm{out}}, p_a,\mathcal{G}_a^u,P_{\bar a,n},\mathcal{G}_{\bar a,n}^{u(i)},\xi)$ and $(\Lambda_{\bar a,n}^m, Q_{\bar a,n}^m, \mathcal{F}_{\bar a,n}^u,P_{\bar a,n},\mathcal{F}_{\bar a,n}^{s(i)},\xi)$ each of which corresponds to $(\Lambda_t, p_t,\mathcal{F}_t^u,\hat p_t,\mathcal{F}_t^{u(k_0+j)},\sigma)$ in Subsection \ref{s_accompany} (see also Fig.\ \ref{fg_1}).
Then, there exists an integer $i_0=i_0(n)>0$ such that, for any $i\geq i_0$,
$$
\left\vert \frac{d \bar y(S_{\bar a,n})}{d \bar a}-\frac{d \bar y(f_{\bar a,n}^{(i)})}{d \bar a}\right\vert <\frac{1}{10}\quad\mbox{and}\quad
\left\vert \frac{d \bar y(L_{\bar a,n}^u)}{d \bar a}-\frac{d \bar y(g_{\bar a,n}^{(i)})}{d \bar a}\right\vert <\frac{1}{10}.$$
This implies our desired inequality (\ref{ineq}).
\end{proof}

Here, we consider the case when $f_{\bar a,n}^{(i)}$ and $g_{\bar a,n}^{(i)}$ have a quadratic tangency in $\xi$.
Then, Lemma \ref{velocity} shows that the tangency unfolds generically with respect to $\psi_{\bar a,n}$ and hence to $\varphi_a$ in the sense of Definition \ref{d_unfold}.

\section{Proof of Theorems \ref{thm_A} and \ref{thm_B}}\label{S_5}

In this section, we will work with the notation as in Section \ref{S_4} and Convention \ref{conv_1}.
Theorems \ref{persistent tangency} and \ref{l_0} below imply respectively the assertions (i) and (ii) of Theorem \ref{thm_A}.

\begin{theorem}\label{persistent tangency}
For any $b\in I\setminus \{0\}$, there exists a closed subinterval $J_b$ of $J_{b,1}$ such that the one-parameter family $\{\varphi_a\}_{a\in J_b}$ with $\varphi_a=\varphi^w_{a,b}$ has generically unfolding persistent quadratic tangencies of $W^u(\Lambda^{\mathrm{out}}_a)$ and $W^s(\Lambda^m_{\bar a,n})$.
\end{theorem}

Note that $\Lambda^m_{\bar a,n}$ corresponds to the basic set $\Lambda^{\mathrm{in}}_a$ in the statement of 
Theorem \ref{thm_A}.

\begin{proof} 
We will work with $\bar \delta$, $\bar V$, $n$ and $i$ satisfying the conclusion of Lemma \ref{velocity}.
By the Intermediate Value Theorem, there exists an $\bar a\in \bar J_{\bar\delta}$ such that the locally highest leaf of $\mathcal{G}^{u(i)}_{\bar{a},n}$ in $\bar V$ and that of $\mathcal{F}^{s(i)}_{\bar{a},n}$ has a tangency in $\bar V$, see Figure \ref{fg_9}-(i).
\begin{figure}[hbt]
\begin{center}
\scalebox{0.75}{\includegraphics[clip]{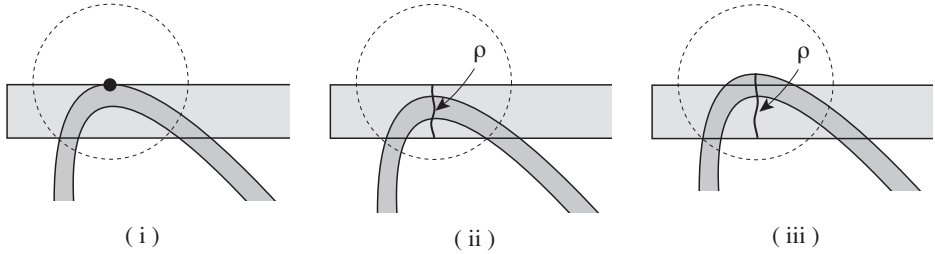}}
\caption{(ii) The case of $\bar a_1>\bar a$.  (iii) The case of $\bar a_1<\bar a$.}
\label{fg_9}
\end{center}
\end{figure}
Then, for any $\bar a_1\in J_{\bar \delta}$ sufficiently close to $\bar a$, there exists an almost vertical $C^1$ arc $\rho$ in $\bar V$ containing subarcs $\rho^u$ and $\rho^s$ which respectively cross $\mathcal{G}^{u(i)}_{\bar{a},n}$ and $\mathcal{F}^{s(i)}_{\bar{a},n}$ exactly, and such that each point of $\rho\cap \mathcal{G}^{u(i)}_{\bar{a},n}\cap \mathcal{F}^{s(i)}_{\bar{a},n}$ is a tangency of leaves in $\mathcal{G}^{u(i)}_{\bar{a},n}$ and $\mathcal{F}^{s(i)}_{\bar{a},n}$, see Figure \ref{fg_9}-(ii) and (iii).
Let $\pi^u: \varphi^i_{a_1}(\mathcal{G}^u_{a_1})\to \rho^u\subset \rho$ and $\pi^s: \psi^{-i}_{\bar a_1,n}(\mathcal{F}^s_{\bar a_1,n})\to \rho^s\subset \rho$ be the $C^1$ projections along their leaves.
Then, the compositions $\eta^u =\pi^u\circ(\varphi^i_{a_1}|l^s_{a_1}):l^s_{a_1}\to \rho$ and $\eta^s =\pi^s\circ(\psi^{-i}_{\bar a_1,n}|l^u_{\bar a_1,n}):l^u_{\bar a_1,n}\to \rho$ are $C^1$ embeddings onto $\rho^u$ and $\rho^s$ respectively.
Note that each point of $\eta^u(W^u_{\mathrm{loc}}(\Lambda_{a_1}^{\mathrm{out}})\cap l^s_{a_1})$ (resp.\ $\eta^s(W^s_{\mathrm{loc}}(\Lambda_{\bar a_1,n}^m)\cap l^u_{\bar a_1,n})$) is in a $\Lambda_{a_1}^{\mathrm{out}}$-leaf (resp.\ $\Lambda_{\bar a_1,n}^m$-leaf).

By Lemma \ref{thick horseshoe}\,(ii), if necessary replacing $n$ and hence $i$ by greater integers, one can suppose that
$$\tau(\Lambda^m_{\bar a_1,n}\cap l^u_{\bar a_1,n})>\frac{1}{c}.$$
This fact together with Lemma \ref{thick horseshoe}\,(i) shows that
$$\tau(\Lambda_{a_1}^{\mathrm{out}}\cap l^s_{a_1})\cdot\tau(\Lambda^m_{\bar a_1,n}\cap l^u_{\bar a_1,n})>1.$$
Thus, there exists a closed subinterval $\bar J_b$ of $\bar J_{\bar\delta}$ such that the last inequality holds if $\bar a_1\in \bar J_b$.
By invoking Thickness Lemma \cite{N0,PT}, for any $\bar a_1\in \bar J_b$, we have a $\Lambda_{a_1}^{\mathrm{out}}$-leaf in $\mathcal{G}^{u(i)}_{\bar a_1,n}$ and a $\Lambda^m_{\bar a_1,n}$-leaf in  $\mathcal{F}^{s(i)}_{\bar a_1,n}$ admitting a quadratic tangency $q_{\bar a_1}$ in $\bar V$.
By considering a vertical segment $\xi$ in $\bar V$ passing through $q_{\bar a_1}$ and using Lemma \ref{velocity}, one can show that $q_{\bar a_1}$ unfolds generically with respect to $\{\psi_{\bar a,n}\}$ and hence to $\{\varphi_{a}\}$.
Since $\bar J_b\subset \bar J_{\bar \delta}\subset \bar J$, it follows that $J_b=\Theta_n(\bar J_b)$ is our desired subinterval of $J_{b,1}=\Theta_n(\bar J)$.
\end{proof}

Theorem \ref{persistent tangency} presents persistent heteroclinic tangencies associated to two 
basic sets. 
The following theorem presents homoclinic
tangencies associated to the fixed point $p_a$ of $\varphi_a$ 
and so we conclude the assertion (ii) of  Theorem \ref{thm_A}.

\begin{theorem}\label{l_0}
With the notation as above, there exists a dense subset $J'$ of $J_b$ such that, for any $\hat a\in J'$, 
$W^u(p_{\hat a})$ and $W^s(p_{\hat a})$ have a quadratic tangency $q_{\hat a}$ which unfolds generically 
with respect to $\{\varphi_a\}_{a\in J_b}$.
\end{theorem}
\begin{proof}
For any fixed $a_1\in J_b$,  
there exists a heteroclinic quadratic tangency of a $\Lambda^{\mathrm{out}}_{a_1}$-leaf of $\mathcal{G}^{u(i)}_{\bar a_1,n}$ and a $\Lambda^m_{\bar{a}_1,n}$-leaf of $\mathcal{F}^{s(i)}_{\bar a_1,n}$ in $\Bar{V}$ unfolding generically with respect to $\varphi_a$.
By Lemma \ref{velocity}, there exists an $a_2\in J_b$  arbitrarily close to $a_1$ such that $W^u(\Lambda^{\mathrm{out}}_{a_2})$ and $W^s(\Lambda^m_{\bar{a}_2,n})$ have a transverse intersection $z_1$ in $\bar V$, see Fig.\ \ref{fg_10}.
\begin{figure}[hbt]
\begin{center}
\scalebox{0.85}{\includegraphics[clip]{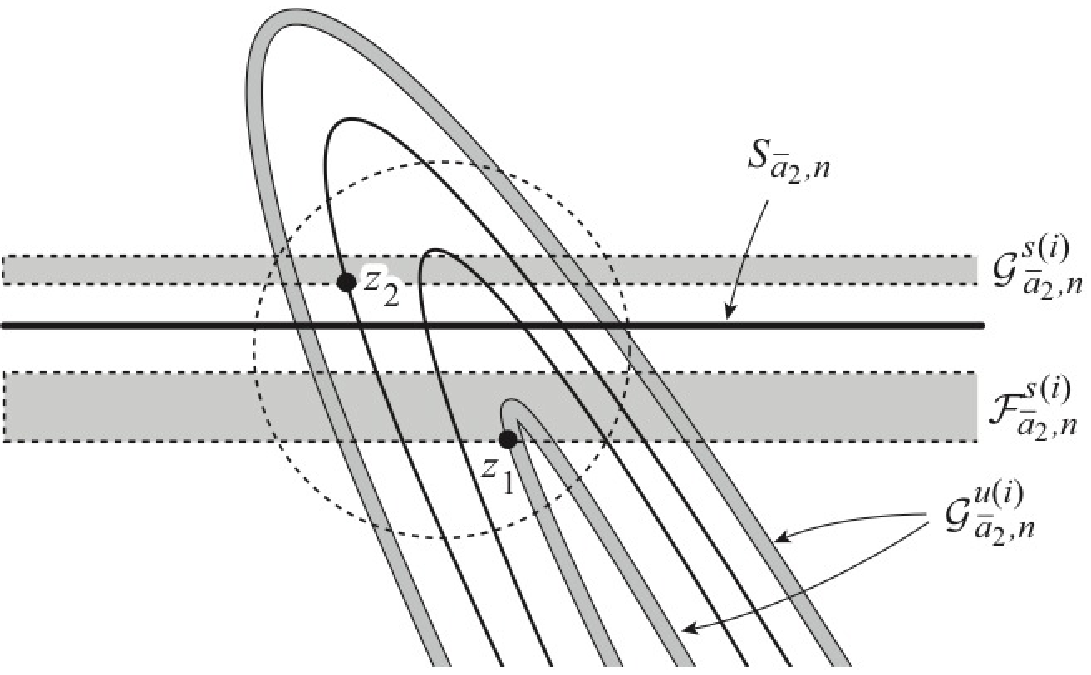}}
\caption{}
\label{fg_10}
\end{center}
\end{figure}

Here, we will show that $W^s(\Lambda^{\mathrm{out}}_{a_2})$ and $W^u(\Lambda^m_{\bar{a}_2,n})$ also have a transverse intersection.
Let $\mathcal{G}^s_{a_2}$ be a foliation compatible with $W^s_{\mathrm{loc}}(\Lambda_{a_2}^{\mathrm{out}})$ such that any leaves of $\mathcal{G}^s_{a_2}$ and $\mathcal{G}^u_{a_2}$ meet transversely in a single point, see Fig.\ \ref{fg_11} (and also Fig.\ \ref{fg_6}).
\begin{figure}[hbt]
\begin{center}
\scalebox{0.85}{\includegraphics[clip]{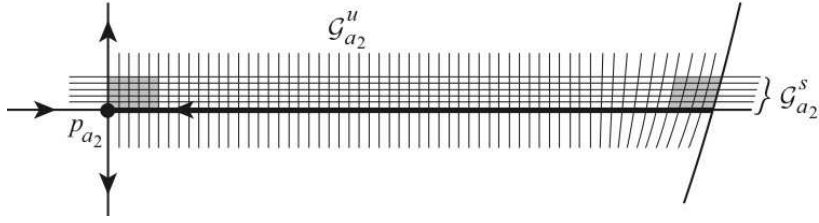}}
\caption{The union of the shaded regions contains $\Lambda_{a_2}^{\mathrm{out}}$.}
\label{fg_11}
\end{center}
\end{figure}
Again by \cite[p.\ 125, Proposition 1, p.\ 129, Remark 1]{PT} and the Inclination Lemma, for any sufficiently large integer $i>0$, one can have a foliation $\mathcal{G}^{s(i)}_{\bar{a}_2,n}$ obtained by shortening the leaves of $\varphi^{-i}_{a_2}(\mathcal{G}^s_{a_2})$ so that all leaves  of $\mathcal{G}^{s(i)}_{\bar a_2,n}$ are well approximated by the arc $S_{\bar a_2,n}$ given in Section \ref{S_4}.
Then, as is shown in Figure \ref{fg_10}, we have a transverse intersection point $z_2$ of a $\Lambda^{\mathrm{out}}_{a_2}$-leaf of $\mathcal{G}^{s(i)}_{\bar a_2,n}$ and a leaf of $W^u(\Lambda^m_{\bar{a}_2,n})$.
Applying \cite[Lemma 8]{N1} to the cycle $\{\Lambda^{\mathrm{out}}_{a_2},z_1,\Lambda^m_{\bar{a}_2,n},z_2\}$, one can define a basic set $\Lambda_{a_2}$ with $\Lambda_{a_2}\supset \Lambda^{\mathrm{out}}_{a_2}\cup \Lambda^m_{\bar{a}_2,n}$.

Since $a_2$ is an element of $J_b$, by Theorem \ref{persistent tangency}, $W^u(\Lambda^{\mathrm{out}}_{a_2})$ and $W^s(\Lambda^m_{\bar{a}_2,n})$ have a \emph{heteroclinic} quadratic tangency $q_{a_2}$ unfolding generically.
Since the basic sets $\Lambda^{\mathrm{out}}_{a_2}, \Lambda^m_{\bar{a}_2,n}$ have dense subsets consisting of saddle periodic points, 
there exists $a_2'\in J_b$ arbitrarily close to $a_2$ and satisfying the following two conditions.
\begin{itemize}
\item
There exist leaves $l_{a_2'}^u$ of $W^u(\Lambda^{\mathrm{out}}_{a_2'})$ and $l_{\bar{a}_2'}^s$ of $W^s(\Lambda^m_{\bar{a}_2',n})$ which have a quadratic tangency $q_{a_2'}$ unfolding generically and moreover pass through saddle periodic points $\hat p^u_{a_2'}\in 
\Lambda^{\mathrm{out}}_{a_2}$ and $\hat p^s_{a_2'}\in\Lambda^m_{\bar{a}_2,n}$ respectively.
\item
There is a basic set $\Lambda_{a_2'}$ of $\varphi_{a_2'}$ which belongs to a continuation of basic sets based at $\Lambda_{a_2}$.
\end{itemize}
Since $\Lambda_{a_2'}$ is a basic set containing $\Lambda^{\mathrm{out}}_{a_2'}\cup \Lambda^m_{\bar{a}_2',n}$, both $W^u(p_{a_2'})$ and $W^s(p_{a_2'})$ pass through an arbitrarily small neighborhood $U$ of $q_{a_2'}$.
Since $q_{a_2'}$ is a tangency unfolding generically, there exists $a_3\in J_b$ arbitrarily close to $a_2'$ such that $W^u(p_{a_3})$ and $W^s(p_{a_3})$ have a \emph{homoclinic} quadratic tangency $q_{a_3}$ in $U$.
By invoking Accompanying Lemma, one can show that $q_{a_3}$ is also a tangency unfolding generically.
Here, we note that Accompanying Lemma works only for leaves sufficiently close to a leaf passing through saddle periodic points.
In our case, the leaf $l^u_{a_3}$ (resp.\ $l^s_{a_3}$) passes through the saddle periodic points $\hat p^u_{a_3}$ (resp.\ 
$\hat p^s_{a_3}$).
This is our reason for replacing the parameter $a_2$ by $a_2'$.

Recall that $a_1\in J_b$ is taken arbitrarily and $a_3\in J_b$ is arbitrarily close to $a_1\in J_b$.
From this fact, we have a dense subset $J'$ of $J_b$ such that, for any $\hat a\in J'$, 
$W^u(p_{\hat a})$ and $W^s(p_{\hat a})$ have a quadratic tangency unfolding generically.
This completes the proof of Theorem \ref{l_0}.
\end{proof}

Next, we prove Theorem \ref{thm_B} by using results of Theorems \ref{persistent tangency} and \ref{l_0}.

\begin{proof}[Proof of Theorem \ref{thm_B}]
For any $\hat a\in J'$, one can apply the Palis-Takens Renormalization Theory to a small neighborhood of $q_{\hat a}$.
Then, Proposition 3.3 in  Robinson \cite{R} and Lemma 2.2 in Newhouse \cite{N2} imply the existence of an open subinterval $Y_{\hat a}$ of $J_b$ arbitrarily close to $\hat a\in J'$ and such that, for any $a\in Y_{\hat a}$, $\varphi_{a}$ has at least one sink whose basin meets $W^u(p_{a})$ non-trivially.
From the denseness of $J^{\prime}$ in $J_b$ and the arbitrary closeness of $Y_{\hat a}$ to $\hat a\in J'$, we have an open dense subset $A_b^{(1)}$ of $J_b$ such that, for any $a\in A_b^{(1)}$, $\varphi_{a}$ admits at least one sink $r$ with $W^u(p_a)\cap B_r\neq \emptyset$, where $B_r$ is the basin of $r$.
According to Proposition 2.1 in \cite{N2}, if $\varphi_{a}$ had an SRB measure $\nu$ supported by the homoclinic set of $p_a$, then the support $\mathrm{supp}(\nu)$ would coincide with the closure $\mathrm{Cl}(W^u(p_{a}))$.
Since $\nu$ is a $\varphi_a$-invariant probability measure, it follows that $\mathrm{Cl}(W^u(p_{a}))\cap B_r\subset \{r\}$.
On the other hand, since $W^u(p_{a})\cap B_r\neq \emptyset$, the intersection would contain an arc, a contradiction.
Thus, for any $a\in A_b^{(1)}$, the homoclinic set of $p_{a}$ does not support any SRB measure.
This proves (i).

By applying arguments as in the proof of Theorem E in \cite[\S 9]{R} repeatedly, one can have open dense subsets $Z_n$ $(n=1,2,\dots)$ of $J_b$ with $Z_1=A_b^{(1)}$ such that $\varphi_{a}$ has at least $2^n$ sinks for any $a\in Z_n$ associated to the periodic-doubling bifurcation.
Then, $A_b^{(2)}=\bigcap_{n\geq 1}Z_{n}$ is a residual subset of $J_b$ such that $\varphi_{a}$ has infinitely many sinks if $a\in A_b^{(2)}$.
This shows (ii).

According to Wang-Young \cite[Appendix A.2]{WY}, for any $\hat a\in J'$, there exists a subset $X_{\hat a}$ of $J_b$ with 
positive Lebesgue measure and contained in an arbitrarily small neighborhood of $\hat a$ in $J_b$ and such that, for any $a\in X_{\hat a}$, $\varphi_a$ has a strange attractor with an SRB measure.
Again by the density of $J'$ in $J_b$, we have a subset $A_b^{(3)}$ of $J_b$ satisfying the conditions required in (iii) of Theorem \ref{thm_B}.
\end{proof}

\bigskip

\subsection*{Acknowledgments.}
The authors would like to thank the referees for their useful comments and suggestions. 
Also, the first and second  authors would like to thank  Bau-Sen Du and Yi-Chiuan Chen 
for their hospitality during our stay in Institute of Mathematics,  Academia Sinica.  
The first and third authors were partially supported by
Grant-in-Aid for Scientific Research (C) 22540226 and 22540092, respectively; 
the second author was supported by NSC 99-2115-M-009-004-MY2.



\begin{thebibliography}{100}

\bibitem{AS}
V. S. Afrajmovich  and  L. P. Shilnikov,
 On critical sets of Morse-Smale systems,
\textit{  Trans.  Moscow Math. Soc.} \textbf{28}
(1973) 179--212.

\bibitem{BY}
M. Benedicks and L.-S. Young,  
SBR-measures for certain H\'enon maps,
 \textit{Invent.\ Math.} \textbf{112} (1993), 541--576.

\bibitem{FG}
J. E. Forn{\ae}ss and E. A.  Gavosto, 
Existence of generic homoclinic tangencies for H\'{e}non mappings,
\textit{J. Geom.\ Anal.} \textbf{2} (1992), no.\ 5, 429--444.
 
\bibitem{Fr}
J. Franks, 
Differentiably $\Omega$-stable diffeomorphisms, 
\textit{Topology} \textbf{11} (1972), 107--113.

\bibitem{GS90}
S. V. Gonchenko  and  L. P. Shilnikov,
 Invariants of $\Omega$-conjugacy of diffeomorphisms with a structurally unstable homoclinic trajectory 
\textit{Ukr.\ Math.\ J.} \textbf{42} (1990), 134--140.


\bibitem{He}
M. H\'enon, 
A two dimensional mapping with a strange attractor, 
\textit{Comm.\ Math.\ Phys.} \textbf{50} (1976), 69--77.

\bibitem{KKY}
I. Kan,  H. Ko\c{c}ak and J. A. Yorke, 
Antimonotonicity: concurrent creation and annihilation of periodic orbits, 
\textit{Ann.\ of Math.} (2) \textbf{136} (1992), no.\ 2, 219--252.

\bibitem{KS1}
S. Kiriki and T. Soma, 
Persistent antimonotonic  bifurcations  and strange attractors for cubic homoclinic tangencies, 
\textit{Nonlinearity} \textbf{21} (2008) 1105--1140.

\bibitem{MR97}
L. Mora and N. Romero,
Moser's invariant curves and homoclinic bifurcations.  
\textit{Dynam. Systems Appl.}  \textbf{6}  (1997), no.\ 1, 29--41.


\bibitem{MV}
L. Mora and M. Viana, 
Abundance of strange attractors, 
\text{Acta Math.} \textbf{171} (1993), no.\ 1, 1--71. 




\bibitem{N0}
S. Newhouse, 
Diffeomorphisms with infinitely many sinks, 
 \textit{Topology} \textbf{13} (1974), 9--18.
 
 
 \bibitem{N1}
S. Newhouse, 
The abundance of wild hyperbolic sets and non-smooth stable sets for diffeomorphisms, 
 \textit{Publ.\ Math.\ I.H.\'E.S.} \textbf{50} (1979), 101--151.
 
\bibitem{N2}
S. Newhouse, 
New phenomena associated with homoclinic tangencies, 
\textit{Ergodic Theory Dynam. Systems} \textbf{24} (2004), no.\ 5, 1725--1738.

\bibitem{PT}
J. Palis and F. Takens,
 Hyperbolicity and sensitive chaotic dynamics at homoclinic bifurcations, Fractal dimensions and infinitely many attractors, Cambridge Studies in Advanced Mathematics \textbf{35}, Cambridge University Press, Cambridge, 1993.

\bibitem{Pol}
M. Pollicott, 
Stability of mixing rates for Axiom A attractors, 
\textit{Nonlinearity} \textbf{16} (2003), 567--578.


\bibitem{R} 
C. Robinson,
Bifurcation to infinitely many sinks,  
\textit{Comm.\ Math.\ Phys.} \textbf{90} (1983), no.\ 3, 433--459.


\bibitem{Rb}
C. Robinson, Dynamical Systems, Stability, Symbolic Dynamics, 
and Chaos (Studies in Advanced Mathematics), 2nd edn.\ CRC Press, Baton 
Rouge, FL, 1999.


\bibitem{R94}
N. Romero,  
Persistence of homoclinic tangencies in higher dimensions,
\textit{Ergod.\ Th.\ Dynam.\ Sys.}, \textbf{15} (1995), no.\ 4, 735--757.



\bibitem{S}
 S. Sternberg,
On the structure of local homeomorphisms of euclidean $n$-space II, 
\textit{Amer.\ J. Math.} \textbf{80} (1958) 623--631.





\bibitem{T}
F. Takens, 
Partially hyperbolic fixed points, 
\textit{Topology} \textbf{10} (1971), 133--147.


\bibitem{WY}
Q. Wang and L.-S. Young, 
Strange attractors with one direction of instability, 
\textit{Comm.\ Math.\ Phys.}
\textbf{218} (2001), no.\ 1, 1--97.



\end{thebibliography}
\end{document}